\documentclass[11pt]{article}
\usepackage{amssymb}
\usepackage{amsthm}
\usepackage{amsmath}
\usepackage{txfonts}
\usepackage{bbold}
\usepackage{mathtools}
\usepackage{verbatim}
\usepackage{cite}
\usepackage{thmtools, thm-restate} 
\usepackage{enumitem}
\usepackage{graphicx}
\usepackage{caption}
\usepackage{subcaption}

\setlist{topsep=0pt, itemsep=-3pt}

\usepackage{color}

\usepackage[colorlinks=true,citecolor=black,linkcolor=black,urlcolor=blue]{hyperref}

\newcommand{\arXiv}[1]{Preprint available on \url{http://arxiv.org/abs/#1}}

\theoremstyle{plain}
\newtheorem{THM}{Theorem}[section]
\newtheorem*{THM*}{Theorem}
\newtheorem{PROP}[THM]{Proposition}
\newtheorem{LEMMA}[THM]{Lemma}

\newtheorem{COR}[THM]{Corollary}
\newtheorem{CLAIM}[THM]{Claim}

\theoremstyle{definition}

\theoremstyle{definition}

\theoremstyle{remark}

\newcommand{\pr}[1]{\mathbb{P} \left[ #1 \right]}
\newcommand{\er}[1]{\mathbb{E} \left[ #1 \right]}
\newcommand{\var}[1]{\text{Var} \left[ #1 \right]}

\newcommand{\tdots}{\mathrel{{.}\,{.}}\nobreak}

\newcommand{\subs}{\subseteq}

\newcommand{\gnp}{G_{n,p}}
\newcommand{\gn}{G_{n, \frac 12}}
\newcommand{\ndl}{(n, d, \lambda)}
\newcommand{\turan}{\mathrm{ex}}

\newcommand{\mb}{\mathbf}
\newcommand{\Wst}{W_{\mb s,\mb  t}}
\newcommand{\Wstp}{W_{\mb s', \mb t'}}

\newcommand{\ffp}[2]{\left(#1 \right)_{#2}}       
\newcommand\restrict[1]{\raisebox{-.3ex}{$|$}_{#1}}

\def\tk#1{ {#1}}   

\title{Zero forcing number of graphs}
\author{ Thomas Kalinowski \thanks {School of Mathematical and Physical Sciences, University of Newcastle, Callaghan, NSW 2308, Australia. \ \ \ \ \ Email:~thomas.kalinowski@newcastle.edu.au.}
\and
Nina Kam\v{c}ev \thanks{Department of Mathematics, ETH, 8092 Zurich. Email: nina.kamcev@math.ethz.ch.}
\and
Benny Sudakov \thanks{Department of Mathematics, ETH, 8092 Zurich.
Email: benjamin.sudakov@math.ethz.ch.
	}
\date{}
}
\begin{document}
    \vspace*{-55pt}
    {\let\newpage\relax\maketitle}


\vspace{-6pt}

\abstract{
A subset $S$ of initially infected vertices of a graph $G$ is called \emph{forcing} if we can infect the entire graph by iteratively applying the following process. At each step, any infected vertex which has a unique uninfected neighbour, infects this neighbour.
The \emph{forcing number} of $G$ is the minimum cardinality of a forcing set in $G$. In the present paper, we study the forcing number of various classes of graphs, including graphs of large girth, $H$-free graphs for a fixed bipartite graph $H$, random and pseudorandom graphs.
}

\section{Introduction}
    
    Let $G$ be a simple, undirected graph on the vertex set $V$. The \emph{zero forcing process} on $G$ is defined as follows. Initially, there is a subset $S$ of black vertices, while all other vertices are said to be white. At each time step, a black vertex with exactly one white neighbour will \emph{force} its white neighbour to become black. The set $S$ is said to be a \emph{zero forcing set} if, by iteratively applying the forcing step, all of $V$ becomes black.
    The \emph{zero forcing number} of $G$ is the minimum cardinality of a zero forcing set in $G$, denoted by $Z(G)$.  
    Note that given an initial set of black vertices, the set of black vertices obtained by applying the colour-change rule until no more changes are possible is unique.  We will often use the adjective `forcing' instead of `zero forcing'.

	The forcing process is an instance of a propagation process on graphs (in particular, it is a cellular automaton). 
	Such processes are a common topic across mathematics and computer science (see, e.g.,~\cite{cfkr}, \cite{bbm}, \cite{kkt}, \cite{fw}). In other fields (statistical mechanics~\cite{clr}, physics~\cite{adkmz}, social network analysis~\cite{granovetter}), diverse graph processes are used to model technical or societal processes. For an overview of the different models and applications, refer to the book~\cite{bbv}.

	The zero forcing process was introduced in~\cite{bg} and used in~\cite{bbbg} as a criterion for quantum controllability of a system. Independently,  \cite{aimminrank} have introduced it to bound the minimum rank, or equivalently, the maximum nullity of a graph $G$.  Given an $n$-vertex graph $G$, let $M(G)$ denote the maximum nullity over all symmetric real-valued matrices $A$ whose zero-nonzero pattern of the off-diagonal entries is described by the graph $G$. This means that for $i \neq j$, the entry $A_{ij}$ is non-zero if and only if $ij$ is an edge in $G$, whereas the diagonal entries are chosen freely. The minimum rank of $G$ is $n-M(G)$.
	This parameter has been extensively studied in the last fifteen years, largely due to its connection to inverse eigenvalue problems for graphs, singular graphs, biclique partitions and other problems.
         Among several tools introduced to study the minimum rank, the zero forcing number has the advantage that its definition is purely combinatorial. In~\cite{aimminrank}, it was shown that $Z(G)\geq M(G)$ for all graphs $G$. 
        To see this, suppose that $A$ is a matrix whose zero pattern is described by $G$, and $S$ is a zero forcing set of cardinality $|S|$ smaller than the nullity of $A$. This guarantees that we can construct a vector $x \neq 0$ such that $Ax=0$ and $x \restrict{S}= 0$. But then by iteratively applying the forcing step to the components of $x$, we can show $x=0$, which is a contradiction. 
		The minimum rank and forcing number of some specific families of graphs have also been computed in~\cite{aimminrank}. As a simple example, a complete graph $K_n$ on $n$ vertices has $Z(K_n)=M(K_n) = n-1$, whereas the $n$-vertex path $P_n$ has $Z(P_n)=M(P_n)=1$. 
More results on this topic can be found in~\cite{catalog} and~\cite{fh13}.

   Recently, there has been a lot of interest in studying the forcing number of graphs for its own sake, 
and its relation to other graph parameters, such as the path cover number~\cite{gprs}, connected domination number~\cite{acdp}, and the chromatic number~\cite{taklimi}.
	Among others,~\cite{cp} and~\cite{gr} contain upper bounds on the zero forcing number of a graph in terms of its degrees. It is easy to see that the trivial lower bound on the zero forcing number of a graph is $Z(G) \geq \delta -1$. Improving this trivial upper bound, and confirming the earlier conjecture of Davila and Kenter,~\cite{dk}, the authors of~\cite{dks} showed that $$Z(G) \geq \delta + (\delta-2)(g-3),$$ 
for a graph $G$ with minimum degree $\delta \geq 2$ and girth $g \geq 3$ (the \emph{girth} is the length of the shortest cycle in a graph).    
    Our first result substantially improves on this bound, with the  exception of very small values of $\delta$.
    \begin{restatable}{THM}{girth} \label{thm:girth}
        Let $G$ be a graph of girth $g$ with minimum degree $\delta$.
        \begin{enumerate}
            \item If $g=2k+1$ for $k \in \mathbb{N}$, then $Z(G) \geq  e^{-1} \left(\frac{\delta^k}{k+1} - \delta^{k-1} \right) $.
            \item If $g=2k+2$ for $k \in \mathbb{N}$, then $Z(G) \geq 2e^{-1} \left(\frac{\delta^k}{k+1} - \delta^{k-1} \right) .$
        \end{enumerate}
    \end{restatable}

	The crucial ingredient of the proof is an upper bound on the density of a graph which contains no cycles $C_3, C_4, \dots C_{g-1}$. This is an instance of the so-called Tur\'an problem (see, e.g.,~\cite{fs}). Given a graph $H$, we define the Tur\'an number  $\turan(n, H)$ to be the maximum number of edges $e(G)$ over all the $n$-vertex graphs $G$ not containing a subgraph isomorphic to $H$. In general, if a graph $G$ does not contain $H$ as a subgraph, we refer to it as \emph{$H$-free}.
	The Tur\'an numbers of graphs have been extensively studied, and the asymptotic value of $\turan(n,  H)$ is known for all non-bipartite graphs $H$. Denote the complete bipartite graph with vertex classes of order $a$ and $b$ by $K_{a, b}$. A celebrated theorem of K\"ovari, S\'os and Tur\'an~\cite{kst} says that 
	for  $a \leq b$, $\turan(n, K_{a, b}) = O \left(n^{2-\frac 1a} \right)$. This implies that for every bipartite graph $H$, there exists $c_H < 1$ such that $\turan(n, H) = O \left(n^{1+c_H}\right)$.
		Using our approach based on the Tur\'an numbers, we can extend Theorem~\ref{thm:girth} to $H$-free graphs $G$, substantially improving the trivial bound $Z(G) \geq \delta-1$.
	\begin{THM} \label{thm:h-free}
	    Suppose that for a graph $H$ and all $n \geq n_0(H)$, $\mathrm{ex} (n, H) \leq \beta_H n^{1+c_H}$. Let $G$ be an $H$-free graph of minimum degree $\delta \geq 2n_0$. Then
	    $$Z(G) \geq 2^{-1- \frac{2}{c_H}} \left(\frac{\delta}{\beta_H} \right)^{\frac{1}{c_H}}.$$
	\end{THM}

	The authors of~\cite{aimminrank} report that somewhat surprisingly, $M(G)= Z(G)$ for many graphs for which $M(G)$ was known. Our next theorem shows that for most graphs,  $M(G)$ and $Z(G)$ are actually far apart.  
    We consider the random graph model $\gnp$. This is an $n$-vertex graph in which every pair of vertices is adjacent randomly and independently with probability $p$. With an abuse of notation, we write $\gnp$ for the sampled graph, as well as the underlying probability space. 
	The model  $\gn$ is particularly interesting since it assigns the same probability to all the $2^{\binom n2}$ graphs, thus allowing us to make statements about a typical graph. We say that an event in $\gnp $ holds \emph{with high probability} if its probability tends to 1 as $n$ tends to infinity.  The  standard $O$-notation is used for  the  asymptotic  behaviour  of  the relative order of magnitude of two sequences, depending on a parameter $n \rightarrow \infty$.	
	Hall et al.~\cite{hhms} have shown that with high probability, the maximum nullity of a random graph $\gn$ lies between $0.49n$ and $0.86n$. On the other hand, we will show that  the zero forcing number of a typical graph is almost as high as $n$.
      \begin{THM} \label{thm:forcing_gnp}
        Let $p = p(n)$ satisfy $\frac{\log^2 n}{\sqrt{n}} \leq p \leq \frac 23$, then with high probability
        $$Z\left(\gnp \right) = n- \left(2+ \sqrt{2}+o(1) \right) \cdot \frac{\log (np)}{-\log (1-p)}.$$
    \end{THM}   
  \noindent  In particular, for $p = \frac 12$ we have $Z\left(\gn \right) = n- \left(2+ \sqrt{2}+o(1) \right)  \log_2 n$, whereas for $p = o(1)$ the formula simplifies to $ Z\left(\gnp \right) = n- \left(2+ \sqrt{2}+o(1) \right)p^{-1} \log (np)$.

	There is a natural trend in probabilistic combinatorics to explore the possible extensions of results about random graphs to the pseudorandom setting.
	A graph is pseudorandom if its edge distribution resembles the one of $\gnp$.
	There are several formal approaches to pseudorandomness. Here, we will use the one based on the spectral properties of the graph.  The \emph{adjacency matrix} of a graph $G = (V, E)$ with vertex set $V = [n]$ is an $n \times n$ matrix whose entry $a_{ij}$ is 1 if $\{i, j \} \in E$, and $0$ otherwise.  
	    The \emph{eigenvalues of a graph} $G$ are the eigenvalues of its adjacency matrix.  An \emph{$\ndl$ graph} is a $d$-regular $n$-vertex graph in which all eigenvalues but the largest one are at most $\lambda$ in absolute value. If $G$ is an $\ndl$ graph, its largest eigenvalue is $\lambda_1 = d$, and the difference $d-\lambda$ is called the \emph{spectral gap}. It is well known (see, e.g.,~\cite{ks06}) that the larger this gap is, the more closely the edge distribution of a regular graph $G$ approaches that of the random graph with the corresponding edge density.
 We prove a theorem which provides spectral bounds on the zero forcing number of a graph. The lower bound is given in terms of the smallest eigenvalue of $G$, akin to the celebrated result of Hoffman on the independence number~\cite{hoffman}, whereas the previously defined $\lambda$ is used for the upper bound. 
	    \begin{THM} \label{thm:zeroforcing_pseudo}
                Let $G$ be an $\ndl$-graph with  smallest eigenvalue $\lambda_{\min}$. Then
                \begin{enumerate}
                    \item $ \displaystyle Z(G) \geq n \left(1+\frac {2\lambda_{\min}}{d-\lambda_{\min}}  \right).$
                    \item $\displaystyle{  Z(G) \leq n \left(1-\frac {1}{2(d-\lambda)} \log \left(\frac{d-\lambda}{2\lambda +1} \right) \right).}$
                \end{enumerate}
            \end{THM}

\noindent            
The first bound on  $n-Z(G)$ is tight, whereas (ii) is tight up to a multiplicative constant.
            
	The rest of this paper is organised as follows. The second section contains results on $H$-free graphs with a forbidden bipartite graph $H$. In the third section, we asymptotically determine the zero forcing number of $\gnp$. Section 4 contains bounds based on the spectral properties of a graph.

\section{Graphs with forbidden subgraphs}
	In this section, we bound the forcing number of graphs with a forbidden bipartite subgraph. First we consider graphs with large girth $g$ and minimum degree $\delta$, showing that 
      	\begin{enumerate}
            \item  $Z(G) \geq  e^{-1} \left(\frac{\delta^k}{k+1} - \delta^{k-1} \right) $ for $g=2k+1, k \in \mathbb{N}$, and
            \item  $Z(G) \geq 2e^{-1} \left(\frac{\delta^k}{k+1} - \delta^{k-1} \right) $ for $g=2k+2, k \in \mathbb{N}$.
        \end{enumerate}
	
    \begin{proof}[Proof of Theorem~\ref{thm:girth} ]
         We will use the result of Alon, Hoory and Linial~\cite{ahl}, which says that a graph $G_1$ of girth $g$ and average degree $d$ has to satisfy 
	\begin{equation} \label{eq:ahl}
		|V(G_1)| \geq
			 \begin{cases}
			 	(d-1)^k, &g = 2k+1 \\
			 	2(d-1)^k, &g = 2k+2.
			\end{cases}
	\end{equation}
	
		Let $G$ be an $n$-vertex graph with girth $g=2k+1$ and minimum degree $\delta$. The proof for the case of even girth is the same. Let  $S$ be a zero forcing set in $G$ of order $s$. Assume that $s \leq \frac {n}{(k+1)}$, otherwise we can apply \eqref{eq:ahl} to the entire graph $G$ and obtain
        $$s > \frac{n}{k+1} \geq \frac{(\delta-1)^k}{k+1} > e^{-1} \frac{\delta^k}{k+1},$$
        so we are done.
        Starting from $S$, we run the zero forcing process until we have reached a set of black vertices $T$ with $|T|= (k+1) s$.
                Let $U$ be the set of vertices which forced some vertex of $T \setminus S$ during our process. Then, since each vertex can force only one of its neighbours, $|U| \geq ks$. Moreover, all the edges with an endpoint in $U$ lie inside $T$. Denoting the number of edges with both endpoints in $T$ by $e(T)$, we have
                $$e(T) \geq \frac{|U|}{2}\cdot \delta \geq \frac{k|S|}{2}\cdot \delta .$$

                The graph $G[T]$ has average degree at least $\frac{k\delta}{k+1}$. Applying \eqref{eq:ahl} to the graph $G[T]$ gives 
                \begin{align*}
                    (k+1)s & \geq \left(\frac{k \delta}{k+1} -1 \right)^k = \delta^k \left(\frac{k}{k+1} \right) ^k
                        \left(1 - \frac{k+1}{k\delta} \right)^k 
                        \geq  \delta^k e^{-1} \left(1- \frac{k(k+1)}{k\delta} \right) \\
                    s &\geq e^{-1 } \left(\frac {\delta^k}{k+1} - \delta^{k-1} \right),
                \end{align*} 
                as required. For the second inequality, we used
                $\left(\frac{k}{k+1} \right)^k \geq e^{-1}$ and $(1- \alpha)^k \geq 1 - k\alpha$ for $\alpha <1, k \in \mathbb{N}$.
                
    \end{proof}
        It is worth mentioning that already for rather small values of $\delta$, our lower bound exceeds the  value $\delta + (\delta-2)(g-3)$ conjectured in~\cite{dk}. Even for girth five, by taking $k=2$, we obtain $Z(G) \geq \frac 13 \left( \frac{2\delta}{3}-1 \right)^2$, which implies the conjecture of Davila and Kenter for $\delta \geq 22$.
        
	The previous approach will now be used to establish a bound which applies to $H$-free graphs for any bipartite graph $H$. We do not try to optimise the constants in this proof.

    \begin{proof}[Proof of Theorem~\ref{thm:h-free}]
        It will be useful to rearrange our hypothesis on the Tur\'an number of $H$, which is that for $n  \geq n_0$, $\turan(n, H) \leq \beta_H n^{1+c_H}$. Suppose that $G_1$ is an $H$-free graph with average degree $d \geq n_0$. In particular, $G_1$ has at least $n_0$ vertices.  Denoting $n_1 = |V(G_1)|$, the hypothesis gives
        \begin{align*}
             \frac 12 d n_1 &= e(G_1) \leq \beta_H n_1 ^{1+c_H} \\
             \left( \frac{d}{2 \beta_H} \right)^\frac{1}{c_H} &\leq n_1.
        \end{align*}
        The proof reduces to the following claim, which we state formally because we will use it to prove Corollary~\ref{cor:kst}.
     
    	\begin{CLAIM} \label{claim:h-free}
    	Suppose that any $H$-free graph $G_1$ with average degree $d \geq n_0$ satisfies
        \begin{equation} \label{eq:turan}
            |V(G_1)| \geq \left( \frac{d}{2 \beta_H} \right)^\frac{1}{c_H} .
        \end{equation}
    		 Let $G$ be an $H$-free graph with minimum degree $\delta \geq 2n_0$. Then
        	$ \displaystyle Z(G) \geq \frac 12 \left( \frac{\delta}{4 \beta_H} \right)^\frac{1}{c_H} .$
    	\end{CLAIM}
         To see this, let $S$ be a zero forcing set in $G$ of order $s$. Assume that $s < \frac n2$, since otherwise we can apply \eqref{eq:turan} to the entire graph $G$ to get
         $$s \geq \frac n2 \geq  \frac 12 \left( \frac{\delta}{2 \beta_H} \right)^\frac{1}{c_H}  \geq \frac 12 \left( \frac{\delta}{4 \beta_H} \right)^\frac{1}{c_H}.$$ 
          Starting from $S$, we run the zero forcing process until we have reached a set of black vertices $T$ with $|T|= 2s$. Let $U$ be the set of vertices which forced some vertex of $T \setminus S$ during our process. Since each vertex can force only one of its neighbours, $|U| = s$. Moreover, all the edges with an endpoint in $U$ lie inside $T$. Hence
                $$e(T) \geq \frac{|U|}{2}\cdot \delta = \frac{s}{2} \delta.$$
                We conclude that the average degree in $G[T]$ is at least $\frac {\delta}{2}$. Now we can apply \eqref{eq:turan} to $G[T]$.
                \begin{align*}
                    |T| = 2s &\geq \left( \frac{\delta}{4 \beta_H} \right)^\frac{1}{c_H},
                \end{align*}
                as required.
    \end{proof}
    
        Recall that $K_{a, b}$ denotes the complete bipartite graph with parts of order $a$ and $b$. Using the well-known result of~\cite{kst}, we give explicit bounds for $K_{a, b}$-free graphs.
    \begin{COR} \label{cor:kst}
        Let $G$ be a $K_{a, b}$-free graph with minimum degree $\delta \geq 4a-4$. Then 
        $$Z(G) \geq \frac 12 \left(\frac{\delta}{4(b-1)^{\frac 1a}} \right)^{\frac{a}{a-1}}.$$
    \end{COR}
    \begin{proof} 
	The aforementioned theorem of K\"ovari, S\'os and Tur\'an~\cite{kst} 
    	states that for an $n$-vertex $K_{a, b}$-free graph with average degree $d$,	 
	\begin{equation} \label{eq:kst}
		d \leq (b-1)^{\frac 1a} n^{1- \frac 1a} + a-1.
	\end{equation}
	It follows that $K_{a, b} $ satisfies \eqref{eq:turan} with $n_0 = 2a-2$ and $c_H = \frac{a-1}{a}$. For, let $G_1$ be an $n$-vertex $K_{a,b}$-free graph with average degree $d \geq 2a-2$. If the second summand in \eqref{eq:kst} is strictly larger than the first one, then we have $d < 2a-2$, which is a contradiction. Therefore the first summand is larger, so
        \begin{align*}
            d &\leq 2(b-1)^{\frac 1a} n^{1 - \frac 1a}.
        \end{align*}
        Rearranging, we get $
            n \geq \left( \frac{d}{2 (b-1)^{\frac 1a}} \right)^{\frac{a}{a-1}}.$ Our bound on the zero forcing number of $K_{a, b}$-free graphs follows from Claim~\ref{claim:h-free}.
    \end{proof}
    
    \noindent
    Since for $b > (a-1)!$, there are constructions of $K_{a, b}$-free graphs with minimum degree $\delta$ and $O \left( \delta^{\frac{a}{a-1}} \right)$ vertices (see, e.g., 
  ~\cite{ars}), the result of Corollary~\ref{cor:kst} is tight up to a constant factor.

\section{The random graph}

	Here we prove that for $\frac{\log^2 n}{\sqrt{n}} \leq p = o(1)$, with high probability
     $$Z\left(\gnp \right) = n- \left(2+ \sqrt{2}+o(1) \right) \cdot \frac{\log (np)}{p}.$$
	Since for $p = o(1)$, $-\log (1-p) = (1+o(1))p$, this establishes the bound of  Theorem~\ref{thm:forcing_gnp}. Restricting to $p=o(1)$ in this section keeps the calculations clearer. The case of constant $p$ is easier (as we explain below) and can be proved similarly. 
    In what follows, we mostly omit floor and ceiling signs for the sake of clarity of presentation. The logarithms are in base $e$ unless stated otherwise. All the inequalities will hold for large enough $n$. In a graph $G$, we denote the set of edges between vertex sets $S$ and $T$ by $E(S, T)$. We write $u \sim v$ if the vertices $u$ and $v$ are adjacent in $G$, and $u \nsim v$ otherwise.
        
	Our approach combines the first and the second moment method, and is somewhat similar to the argument used to determine the independence number of $\gnp$ (see, e.g.~\cite[Section 7]{jlr}). However, the proof contains some delicate points, which we try to explain before giving the formal argument. We start this discussion by considering only the uniform model $\gn$.

	The zero forcing number of the random graph is governed  by the occurrence of a specific substructure called a witness. In a graph $G$ on the vertex set $V$, a {\em $k$-witness} (or a {\em witness} of order $k$) is a pair of ordered vertex $k$-tuples  $\left((s_i)_{i \in [k]}, (t_i)_{i \in [k]} \right)$ such that $s_i, t_i \in V$,  $s_i \sim t_i$ for each $i$, and $s_i \nsim t_j$  for $i< j$.    The definition requires $s_i \neq s_j$ for $ i \neq j$, but  it might happen that $s_i = t_j$ for some $i>j$. The adjacency matrix of a $k$-witness, where the rows and columns are indexed by $(s_i)$ and $(t_i)$ respectively, can be found in Figure~\ref{fig:witness_a}.
For any pair of $k$-tuples $\left((s_i)_{i \in [k]}, (t_i)_{i \in [k]} \right)$, we define the set of \emph{superdiagonal} pairs to be $\left \{ \left(s_i, t_j \right): i, j \in [k], i<j \right \}$, and the set of \emph{diagonal pairs} to be $\left\{ \left(s_i, t_i \right): i \in [k] \right \}$.
It is easy to see that if $G$ has a forcing set $S$ of order at most $n-k$, then it contains a $k$-witness (Lemma~\ref{lemma:witness}). Therefore, our aim is to find the order of the largest witness in $\gn$. The first subtlety arises in trying to use the first moment computation to guess the forcing number of $\gn$. Let $R_k$ be the random variable counting the $k$-witnesses in $\gn$. For fixed $k$-tuples $(s_i)$ and $(t_i)$, the probability that they form a witness is $2^{-\binom{k+1}{2}}$ since they determine $\binom{k+1}{2}$ superdiagonal and diagonal pairs. Using linearity of expectation, we can multiply this probability with the number of choices for $(s_i)$ and $(t_i)$ to obtain
            \begin{align*} 
                \er{R_k} 
				& \leq \left( \frac{n!}{(n-k)!} \right)^2 \cdot 2^{-\binom{k+1}{2}} 
                \leq  n^{2k} 2^{-\frac{k^2}{2}} =
                 \left ( n^2 2^{-\frac{k}{2}} \right)^k. \label{eq:expectationbound}
            \end{align*}

 For $k\geq (4+\epsilon) \log_2 n$, it holds that $\er{R_k} = o(1)$, so Markov's inequality implies that with high probability $\gn$ contains no  such $k$-witness. On the other hand, if $k \leq (4-\epsilon) \log_2 n$ the expected number of $k$-witnesses tends to infinity with $n$, so it is reasonable to take $(4+o(1)) \log_2 n$ as the first guess for the order of the largest witness.  For example, in finding the independence number of $\gnp$, the first moment computation gives the correct value throughout the range of $p$. However, in our case, the actual asymptotic order of a witness, $k_c := \left( 2+\sqrt{2} \right)\log_2 n$, is smaller than $4 \log_2 n$. 

The reason for this is that a substructure of a $k$-witness $\left((s_i)_{i \in [k]}, (t_i)_{i \in
    [k]} \right)$ obtained by discarding a final segment of $(s_i)$ and an initial segment of
$(t_i)$ has a lower expected number of copies inside $\gn$ than the witness itself, and therefore
gives a stronger bound on $k$. The correct asymptotic value is obtained by taking \tk{$\ell  =\ell(k)=
\tfrac{\sqrt 2}{2}k$}, and counting substructures called $k$-subwitnesses. The adjacency matrix of a subwitness is depicted \tk{in} Figure~\ref{fig:witness_b}.
 For integers $a \leq b$, define the interval $[a \tdots b] = \{a, a+1, \dots, b \}$. A {\em k-subwitness} in $G$ is a pair of $\ell$-tuples
        $\left((s'_i)_{i \in [\ell]}, (t'_i)_{i \in [k-\ell +1 \tdots k]} \right)$ such that $s'_i, t'_j \in V$, $s'_i \sim t'_i$ for each $i \in [k- \ell + 1 \tdots \ell]$, and $s'_i \nsim t'_j$  for $ 1 \leq i< j \leq k$.  
  Clearly, if $\left((s_i)_{i \in [k]}, (t_i)_{i \in [k]} \right)$ is a witness in $G$, then the restrictions $(s_{1}, \dots s_{\ell})$ and $(t_{k-\ell+1}, \dots, t_k)$ form a subwitness.  
           
In Lemma~\ref{lemma:gnp_fmm}, we will count $k$-subwitnesses for $k = (1+\epsilon)k_c = (1+\epsilon) \left(2+\sqrt{2} \right)\log_2 n$, and show that with high probability, $\gn$ contains no such $k$-subwitness. An intuition on why we can discard those segments of $(s_i)$ and $(t_i)$ is that given a $(1-\epsilon)k_c$-subwitness in $\gn$, we can extend it to a $(1-\epsilon)k_c$-witness with high probability.
  Indeed, to find the remaining columns $t_1, \dots, t_{k-\ell}$, we only need to consider their adjacencies with $s_1, \dots, s_{k-\ell}$. But $k-\ell = (1-\epsilon) \log_2 n$, and therefore a short computation shows that for any choice of $s_1, \dots, s_{k-\ell}$, with high probability, there will be at least $n^{\frac{\epsilon}{2}}$ vertices in $\gn$ satisfying any adjacency restriction with $s_1, \dots, s_{k-\ell}$.  Of course, the same property of $\gn$ implies that we can find the missing rows $s_{\ell +1}, \dots, s_k$.

\begin{figure}[bht]
	\centering
    \begin{subfigure}[b]{0.33\textwidth}
		\raisebox{-6cm}{\includegraphics[scale=.36]{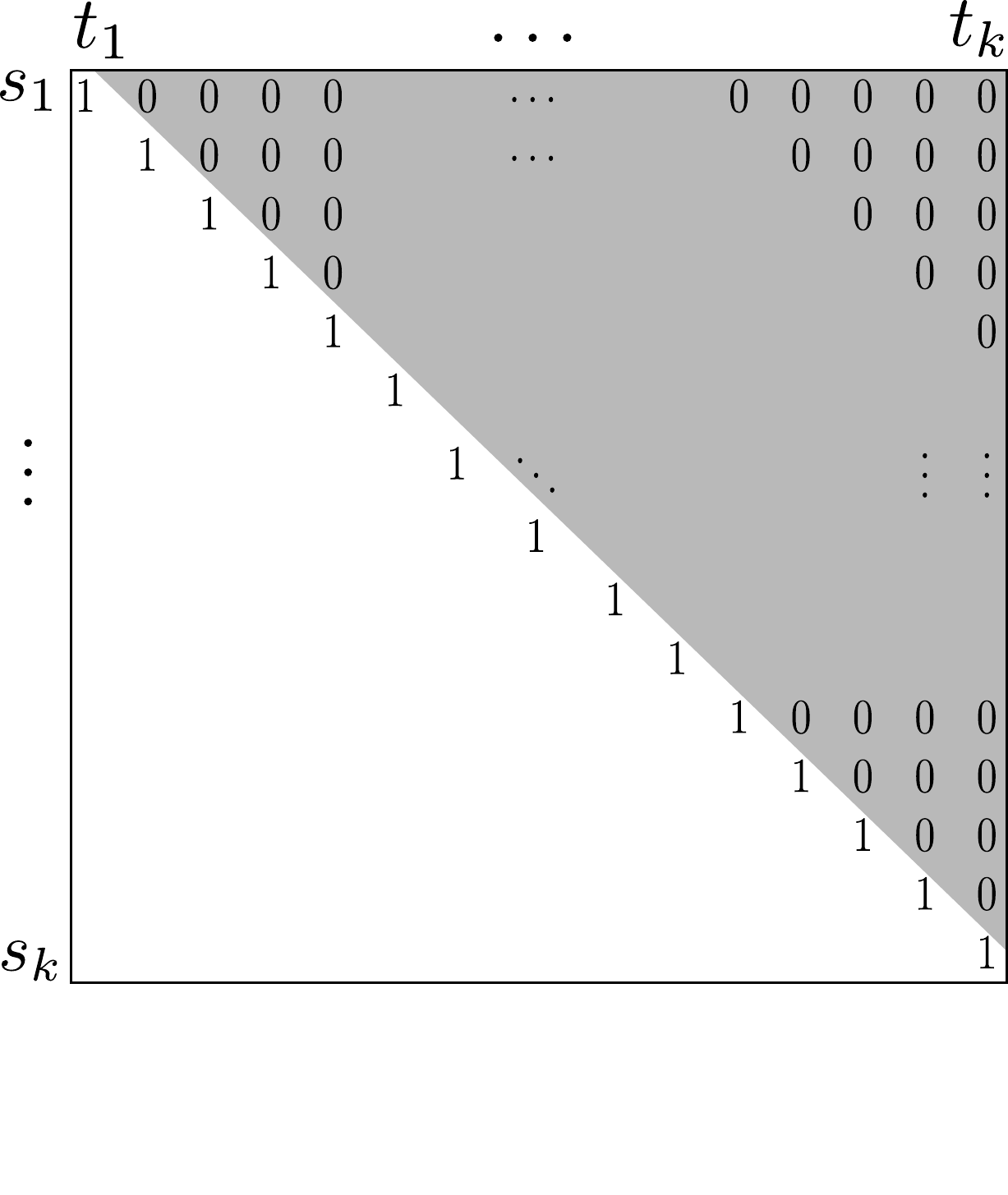}}
        \caption{}
        \label{fig:witness_a}
    \end{subfigure}
	\hspace{-20pt}
    \begin{subfigure}[b]{0.33\textwidth}
		\raisebox{-2cm}{\includegraphics[scale=.36]{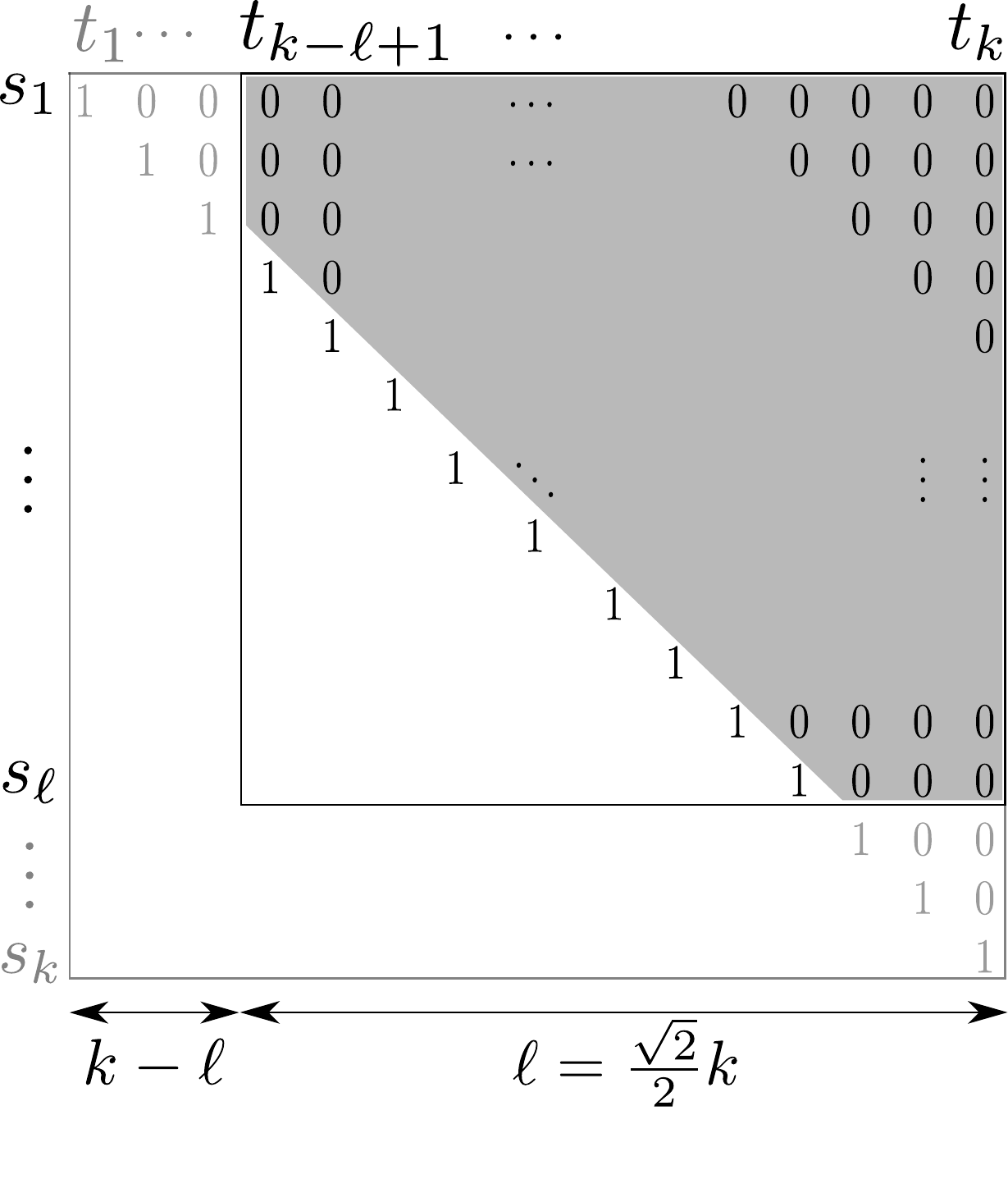}}
        \caption{}
        \label{fig:witness_b}
    \end{subfigure}
	\hspace{-25pt}
    \begin{subfigure}[b]{0.33\textwidth}
		\raisebox{0cm}{\includegraphics[scale=.36]{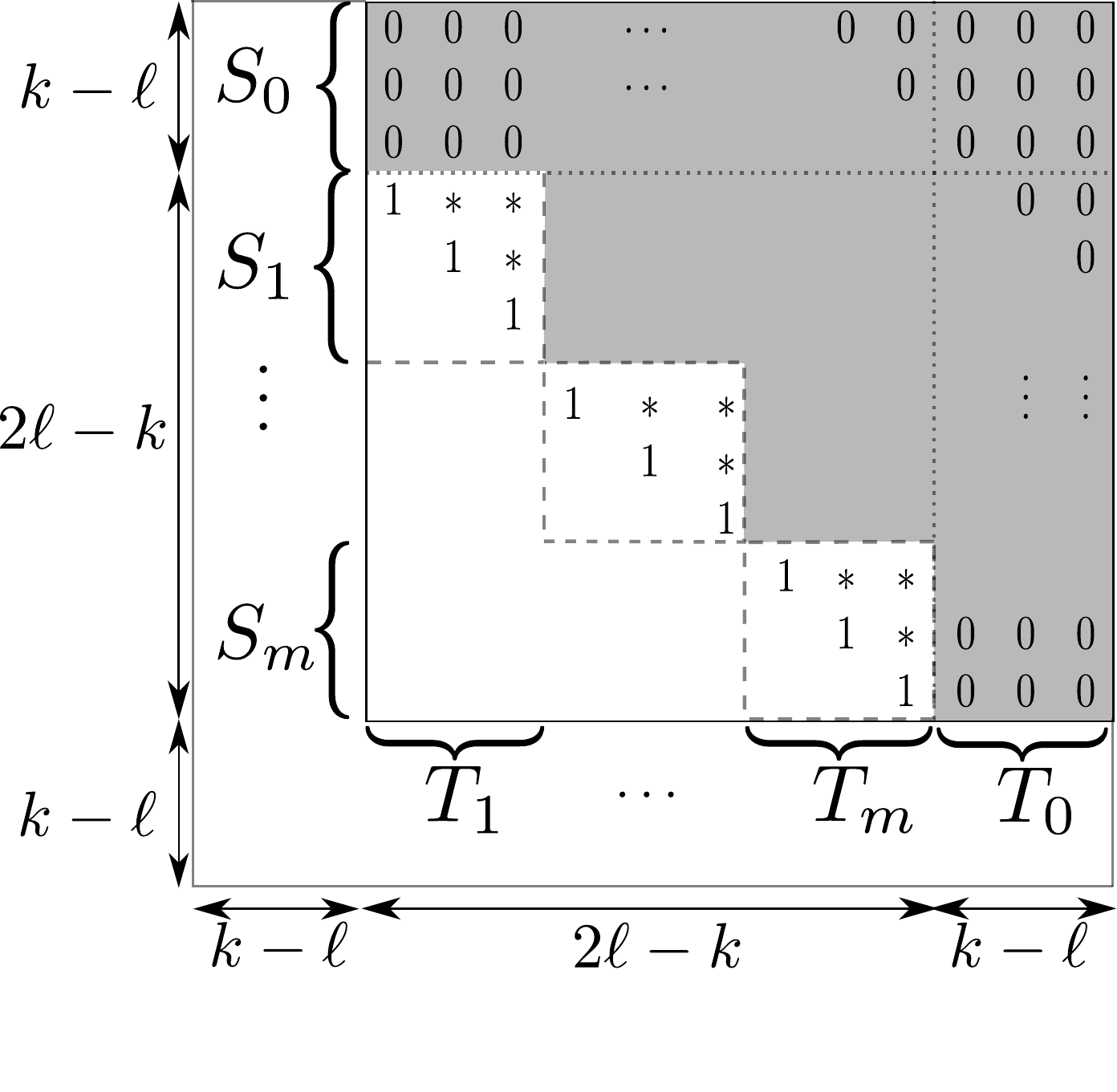}}
        \caption{}
        \label{fig:witness_c}
    \end{subfigure}
	\caption{{Adjacency matrices of a {\em k}-witness, {\em k}-subwitness and a loose {\em k}-subwitness respectively. The regions which are required to contain only zeros are shaded. In Figure (c), the stars mark the entries which are superdiagonal in this particular ordering of the vertices, but not required to contain zeros.}}
\end{figure}

This intuition gives the correct answer for $\gn$, but when $p=o(1)$ the computation of the expected number of  $k$-subwitnesses contains another caveat. In a subwitness, $(s'_i)_{i \in [\ell]}$ and $(t'_i)_{i \in [k-\ell +1 \tdots k]}$ are orderings of the corresponding vertex sets. In the first moment computation, the ordering contributes a factor of $(\ell!)^2 = k^{2\ell + o (k)}$.
This factor was negligible when $p = \frac 12$ and $k_c = \Theta(\log n)$, but for $p=o(1)$, $\gnp$ contains witnesses whose order is polynomial in $n$, so we need to be more careful. Next, we explain how to shave off  a factor of $k^{k+o(k)}$. A subwitness is modified so that the adjacency matrix is invariant under reordering large subsets of the vertices, at the price of discarding a small number of its zero entries. Figure~\ref{fig:witness_c} illustrates this tradeoff.

In a graph $G$, we define a \emph{loose $k$-subwitness} (or just \emph{loose subwitness}) to be a substructure labelled by sets $S_0, S_1, \dots S_m, T_0, T_1, \dots, T_m \subs V$ and bijections $f_i: S_i \to T_i$ for $i = 1, \dots m$, where $m =(2 \ell-k)p = \left( \sqrt{2}-1 \right)kp$ and the following conditions are satisfied. 
\begin{enumerate}
\item The sets $S_i$ are pairwise disjoint for $i = 0, \dots, m$, as well as the sets
  $T_i$. Denoting $S = \bigcup_{i=0}^m S_i$ and $T = \bigcup_{i=0}^m T_i$, we have
  $|S| = |T| = \ell$ and $|S_0|= k - \ell = \left(1 - \frac{\sqrt{2}}{2} \right)k$.
  The sets $S \setminus S_0$ and $T \setminus T_0$ are partitioned \emph{equitably} into $S_i$, and
  $T_i$, that is, $ 0\leq |S_i|-|S_j| \leq 1 $ for $1 \leq i < j \leq m$.  Therefore,
  $ \frac 1p -1 \leq |S_i| = |T_i| \leq \frac 1p +1 $ and the orders of $S_i$ are non-\tk{in}creasing for
  $i \in [m]$.
\item For the edges between $S$ and $T$, we require $E(S_0, T) = E(T_0, S) = \emptyset$, and
  $E(S_i, T_j) = \emptyset$ for $i<j$. For each bijection $f_i$ and $v \in S_i$, there is an edge in
  $G$ between $v$ and $f_i(v)$. In other words, the bijections $f_i$ determine a matching between
  $S_i$ and $T_i$.
\end{enumerate}

 The key point is that any graph that contains a $k$-witness, also contains a loose $k$-subwitness, which we use to get an improved bound. Since the sets $S_i$ and $T_i$ are not ordered, but only paired by the bijections $f_i$ for $i \in [m]$, we gain a factor of $k^{-k +o(k)}$ in the first moment computation. On the other hand, for $i \in [m]$,  we do not care about the adjacency relation between $S_i$ and $T_i$  apart from the diagonal vertex pairs, but that costs us a much less significant factor $ (1-p)^{|S_i|^2m} = (1-p)^{ (2\ell -k)/p}$. 
In words, the definition of a loose subwitness allows a large number of vertex permutations. However, since those permutations preserve most of the superdiagonal pairs, the adjacency matrix of a loose subwitness still contains most of the zeros which were previously required.

 Estimating the expected number of loose subwitnesses in $\gnp$, we can match the bound obtained from the second moment method (Lemma~\ref{lemma:gnp_smm}). This computation requires some additional understanding of how two witnesses can interact, but the ideas are explicit in the argument.
We are now ready to present\ the formal proof. We first establish the relationship between the zero forcing number of a graph and the occurrence of a witness. We will be using the shortened notation $\mb s = (s_i)_{i \in [k]}$, and denote the image of this $k$-tuple by $\mb s [k] =  \{s_i: i \in [k]\}$. 
   
    \begin{LEMMA} \label{lemma:witness}
        Let $G$ be an $n$-vertex graph. If $Z(G) \leq n-k$, then $G$ contains a $k$-witness. Moreover, if $G$ contains a $k$-witness $\left( \mb s, \mb t \right)$ with $\mb s [k]  \cap \mb t [k]  = \emptyset$, then $Z(G) \leq n-k$.
        
    \end{LEMMA}
        \begin{proof}
            Assume that $Z(G) \leq n-k$, that is, $G$ has a forcing set $S$ with $n-|S| \geq k$. Index the vertices of $V(G) \setminus S$ according to the order in which they  were forced, so $t_1$ is the first forced vertex, $t_2$ the second, and so on, up to $t_k$. For $1 \leq i \leq k$, let $s_i$ be a vertex which forced $t_i$. Then by the definition of a zero forcing set, $\left((s_i)_{i \in [k]}, (t_i)_{i \in [k]} \right)$ is a witness.
            
            Conversely, let  $\left( \mb s, \mb t \right)$ be a $k$-witness with $\mb s [k]  \cap \mb t [k]  = \emptyset$. Then $V(G) \setminus \mb t [k]  $ is a forcing set in $G$, since the vertices $t_1, \dots, t_k$ can be forced by the vertices $s_1, \dots, s_k$ respectively.
        \end{proof}           
        
        We now formalise the ideas outlined in the previous discussion.
 
    \begin{LEMMA} \label{lemma:gnp_fmm}
        Let $\frac{C}{n} < p(n) <1$ for a large constant $C$, and define $k = \frac{\left(2+ \sqrt{2} \right)}{p} \left(\log (np) + \log \log (np)\right)$. With high probability, $\gnp$ contains no $k$-witness, and therefore $Z \left( \gnp \right) \geq n-k$.
    \end{LEMMA}
    
    \begin{proof}
        For a graph $G$ on the vertex set $V$, a $k$-witness and a loose $k$-subwitness have been defined in the previous section. We set $k = \frac{\left(2+ \sqrt{2} \right)}{p} \left(\log (np) + \log \log (np)\right)$ and $\ell = \frac{\sqrt{2}}{2}k$, $r = 2\ell -k=  \frac{\left(\sqrt{2} +o(1) \right)}{p} \log (np)$, $m = pr = \left(\sqrt{2} +o(1) \right) \log (np)$. The crucial fact is that if $G$ contains a witness $\left((s_i)_{i \in [k]}, (t_i)_{i \in [k]} \right)$, then a loose subwitness can be found as follows. $S_0$ consists of the first $k - \ell$ rows of the witness, $S_0:= \{s_1, \dots, s_{k-\ell} \}$, and $T_0$ of the last $k - \ell$ columns, $T_0:= \{t_{\ell+1}, \dots, t_k \}$.  The sets $S_1, \dots, S_m$ are constructed by ordering the vertices $s_{k-\ell+1}, s_{k-\ell+2}, \dots, s_{\ell}$ and partitioning them into $m$ equitable intervals. Naturally, $T_1, \dots, T_m$ are the corresponding columns, $T_j = \{t_i: s_i \in S_j \}$, and the bijections $f_j$ map $s_i$ to $t_i$ for $i \in [k-\ell +1 \tdots \ell]$.

        Let $Y_k$ denote the number of loose $k$-subwitnesses in $\gnp$. Our aim is to show
        $\er{Y_k}\tk{\to 0}$. Fix the sets $S_j, T_j$ and bijections $f_j$  which satisfy (i). In particular, $ |S_1 \cup \dots \cup S_m| = |T_1 \cup \dots \cup T_m| = 2\ell - k = r$  and $|S_i | = |T_i| \in \left[\frac 1p -1, \frac 1p+1 \right]$ for $i \in [m]$.   For $S_j$, $T_j$ and $f_j$ to span a loose subwitness, i.e.~for (ii) to be satisfied, we require
        $ \left(\ell^2 - r ^2 \right) +  \left( \binom{r }{2}-m \cdot \binom{r /m}{2} \right)$
        pairs to be non-edges in $\gnp$. The first summand, $\ell^2 - r^2 = 2(k-\ell)\ell - (k-\ell)^2,$ comes from $E(S_0, T) = E(T_0, S) = \emptyset$, and the second from  $E(S_i, T_j) = \emptyset$ for $i<j$. The $m \cdot \binom{r /m}{2}$ pairs have been subtracted since we do not impose any restrictions on $E(S_i, T_i)$, and they will turn out to be negligible.
        Moreover, (ii) requires $r$ diagonal edges to be present in $\gnp$, so the probability that our fixed $S_j, T_j$ and bijections $f_j$ satisfy (ii) is at most
        $$p^r (1-p)^{\ell^2-\frac{r^2}{2} - \frac{r^2}{\tk{2m}}}.$$
        
        Now we will take the union bound over all the potential loose subwitnesses. There are at most $\binom{n}{\ell-r}^2$ choices for $S_0, T_0$, and  $\binom{n}{\frac rm }^m$ choices for the sets $S_j$, $j \in [m]$. Each vertex $v \in S_j$ gets assigned a \tk{vertex} $f_j(v)$, which can be done in at most $n^r$ ways. This assignment also determines the sets $T_j = f_j(S_j)$ for $j \in [m]$.

        \begin{align*}
            \er{Y_k} & \leq \binom{n}{\ell-r}^2 \binom{n}{\frac rm }^m n^r p^r (1-p)^{\ell^2-\frac{r^2}{2} - \frac{r^2}{\tk{2m}}} \leq \left(\frac{en}{\ell-r}  \right)^{2\ell-2r}\left( \frac{enm}{r} \right)^r n^r p^r e^{-p \left(\ell^2-\frac{r^2}{2} - \frac{r^2}{\tk{2m}} \right)}.
        \end{align*}
            We use the inequalities $\ell-r \geq \frac k4$ and $r \geq \frac k4$ to obtain
        \begin{align*}
            \er{Y_k} & \leq \left( \frac{4en}{k} \right)^{2\ell -2r} \left( \frac{4en}{k} \right)^{r}m^rn^rp^r e^{-p \left(\ell^2-\frac{r^2}{2} - \frac{r^2}{\tk{2m}} \right)} 
             \leq C_1^k (np)^{2\ell}(kp)^{-2 \ell +r} m^r  e^{-p \left(\ell^2-\frac{r^2}{2} \right)},
        \end{align*}
        where $C_1\tk{=4e^2}$. 
        For the second inequality, we used $m = pr$, so that
            $e^{\frac{pr^2}{\tk{2m}}} \tk{=}e^{\tk{r/2}} \leq e^k$. Finally, \tk{substituting} $\ell =
             \frac{\sqrt{2}}{2}k$,  $\ell^2 - \frac{r^2}{2}= (\sqrt{2}-1)k^2$, and noting that $(kp)^{-2\ell+r} m^r \leq
              (kp)^{-2 \ell+2r}<1$, we obtain
        \begin{align*}
            \er{Y_k} < C_1^k \left( np \right)^{\sqrt{2}k}e^{-pk^2 \left( \sqrt{2}-1\right)}.
        \end{align*}
            Taking $np$ sufficiently large and recalling that $pk = \left(2+\sqrt{2} \right) \left(\log(np)+ \log \log (np) \right)$, we get $\er{Y_k} < 2^{-k} $.
             With high probability, $\gnp$ contains no loose $k$-subwitnesses, and therefore no $k$-witnesses. By Lemma~\ref{lemma:witness}, $Z(G) \geq n-k$ with high probability.
    \end{proof}
        
	To show that with high probability, $\gnp$ contains a $k$-witness $(\mb s,  \mb t)$ with $ \mb s[k] \cap \mb t[k]=\emptyset$, we use a well-known consequence of Chebyshev's inequality.   
Let $X_n$ be a sequence of non-negative random  variables indexed by some parameter $n$ going to infinity. If $\er{X_n} \rightarrow \infty$ and $\frac{\var{X_n}}{\left(\er{X_n} \right)^2} \rightarrow 0$, then with high probability $X_n>0$. The proof can be found for example in~\cite[Corollary 4.3.2]{alonspencer}. 

    \begin{LEMMA} \label{lemma:gnp_smm}
        Let $p = p(n)$ satisfy $ \frac{\log^2 n}{\sqrt{n}} < p = o(1)$ and $\epsilon > 0$. With high probability
        $$Z\left(\gnp \right) \leq n- \left( 1- \epsilon \right) \cdot \frac{\left(2+ \sqrt{2} \right) \log (np)}{p} .$$
    \end{LEMMA}

     \begin{proof}
        Partition the vertex set $V$ of $\gnp$ into $V_1$ and $V_2$ with $|V_1| = \lfloor \frac n2 \rfloor$ and $|V_2| = \lceil \frac n2 \rceil$. Fix $k = \left( 1- \epsilon \right) \cdot \frac{\left(2+ \sqrt{2} \right) \log (np)}{p} $. We say that a pair $(\mb s, \mb t)$ (or the corresponding $k$-witness) is \emph{divided} if $s_i \in V_1$ and $t_i \in V_2$ for $i \in [k]$. We will show that with high probability, $\gnp$ contains a divided $k$-witness. Let $X_k$ denote the number of such $k$-witnesses in $\gnp$. Furthermore, for a pair of $k$-tuples $(\mb s, \mb t)$, we denote the event that $(\mb s, \mb t)$ is a $k$-witness by $\Wst$.
            Let us first estimate the expectation of $X_k$. We will denote $n' = \frac n2$, and define the falling factorial power
            \begin{equation}
                \ffp{n'}{k} = \frac{n'!}{(n'-k)!}.
            \end{equation}
            We crudely bound $\ffp{n'}{k} \geq \ \left( \frac{n'}{2} \right)^k$ for $k \leq \frac {n'}{2}$. We also use $1-p \geq e^{-(p+p^2)} \geq e^{-1.1p}$ for $0 \leq p \leq 0.1.$ In the following equation, the sum runs over all divided pairs of $k$-tuples $(\mb s, \mb t)$.
            \begin{align*} 
                \er{X_k} &= \sum_{\mb s, \mb t} \pr{\Wst}  = \left (\ffp{n'}{k} \right)^2 p^k (1-p)^{\binom{k}{2}} 
                \geq \left( \frac {n}{4}\right)^{2k}  p^k e^{-1.1p \cdot \frac {k^2}{2} } \\
                 & \geq \left(  \frac{1}{\tk{16}} n^2 p e^{-1.1 \left(1 + \frac{\sqrt{2}}{2} \right)\log(np)} \right)^k 
                  >	 \left(  \frac{1}{\tk{16}} n^2 p \left(np \right)^{-1.9} \right)^k 
            \end{align*}     
            where in the second line, we used $pk \leq \left(2+\sqrt{2} \right) \log (np)$. It follows that
            $\er{X_{k}}  \geq \left(n^{0.1} \right)^k \longrightarrow \infty.$
            
            To use Chebyshev's inequality, we will need second moment estimates.
            Fix a specific divided pair $(\mb s, \mb t)$. The events $\Wstp$ are symmetric over all the divided pairs $(\mb s', \mb t')$, so the standard computation (see, e.g.~\cite[Section 4.3]{alonspencer}) gives 
            \begin{align*}
                \var{X_k} \leq \er{X_k} \sum_{\substack{ s', t' \\ 
                \mb s[k] \cap \mb s'[k]\neq \emptyset \\
                \mb t[k] \cap \mb t'[k] \neq \emptyset
            }}
                \pr{\Wstp \mid \Wst }.
            \end{align*}
            
            Note that the sum includes the case $(\mb s', \mb t') = (\mb s, \mb t)$, so we do not need the additional summand $\er{X_k}$ which often appears in the formula. For any pair of $k$-tuples $(\mb s', \mb t')$, we have defined the set of \emph{superdiagonal} pairs to be $\left \{\left(s'_i, t'_j \right): i, j \in [k], i<j \right \}$, and the set of \emph{diagonal pairs} to be $\left \{\left(s'_i, t'_i \right): i \in [k]\right  \}$.
            We now partition the pairs $(\mb s', \mb t')$. Let $P_{a, b, d}$ denote the set of divided pairs $(\mb s', \mb t')$ such that
            \begin{itemize}
                \item $|\mb s[k] \cap \mb s'[k]|=a$, $| \mb t[k] \cap \mb t'[k]|=b$, and
                \item the number of vertex pairs which are diagonal in both $(\mb s, \mb t)$ and $\left(\mb s', \mb t' \right)$ is $d$.
            \end{itemize}
             Moreover, we define the term $T_{a, b, d} = \frac{1}{\er{X_k}} \sum_{(\mb s', \mb t') \in P_{a, b, d}} \pr{\Wstp \mid \Wst }$.  If $d> a$ or $d> b$, $T_{a, b, d}=0$, so we let $d$ run up to $a$ for simplicity.
			Using this partition, our sum can be written as
			\begin{equation}
				\frac{\var {X_k}}{\left(\er{X_k} \right)^2}  \leq \sum_{a, b = 1}^k \sum_{d=0}^{a} T_{a, b, d}.
			\end{equation}

			We now analyse the term $T_{a, b, d}$. 
We start by counting the divided pairs in $P_{a, b, d}$. There are at most $\ffp{n'-k}{k-a} \ffp{n'-k}{k-b}$ ways of selecting and indexing the vertices of $\mb s'[k] \setminus \mb s[k]$ and $\mb t'[k] \setminus \mb t[k]$. Then we select $d$ distinct pairs $\left(s_i, t_i \right)$ which are diagonal in $(\mb s, \mb t)$ and will also be diagonal in $(\mb s', \mb t')$, which can be done in at most $\binom{k}{d}$ ways. The number of ways to place those pairs into $(\mb s', \mb t')$, i.e.~to choose the index $j$ such that $s'_j = s_i$ and $t'_j  = t_i$  is at most $\ffp{k}{d}$. 
			Similarly, we choose $a-d$ of the remaining vertices in $\mb s[k]$, and assign them any preimage under $\mb s'$, which gives an additional factor of $\binom{k-d}{a-d} \ffp{k-d}{a-d}$. Finally, we do the same for $\mb t$ and $\mb t'$. Altogether,
			\begin{equation} \label{eq:pabd}
				\left | P_{a, b, d} \right | \leq \ffp{n'-k}{k-a} \ffp{n'-k}{k-b} \binom{k}{d} \ffp{k}{d} \binom{k-d}{a-d} \ffp{k-d}{a-d} \binom{k-d}{b-d} \ffp{k-d}{b-d}.
			\end{equation}
			We will later use the fact that $\frac{\ffp{n'-k}{k-a}}{\ffp{n'}{k}} \leq \frac{\ffp{n'-a}{k-a}}{\ffp{n'}{k}} = \frac{(n'-a)!(n'-k)!}{(n'-k)!n'!}= \frac{1}{\ffp{n'}{a}}$.
			To bound the probability of $(\mb s', \mb t') \in P_{a, b, d}$ being a $k$-witness, we first bound the number of pairs which are superdiagonal for both $(\mb s, \mb t)$ and $(\mb s',\mb  t')$. We can even forget about how the overlap vertices are placed in $\left(\mb s', \mb t' \right)$, that is, the number of common superdiagonal pairs is bounded above by $\Phi(a, b)$, where 
				$$\Phi(a, b) =  \max_{A, B \subs [k],\, |A| = a, \, |B| = b} \left | \left \{ (i, j) \in A \times B : i<j \right \} \right |.$$ 
				
			Note that trivially $\Phi(a, b) \leq ab$. By definition of $\Phi$, $ \pr{ \Wstp \mid \Wst } \leq p^{k-d} (1-p)^{\binom k2 - \Phi(a, b)}$. Dividing by $\er{X_k}$ and using \eqref{eq:pabd}, we get
			\begin{equation} \label{eq:tabd}
				T_{a, b, d} \leq \frac{1}{(n')_a(n')_b} k^{a+b-d} \binom{k}{d} \binom{k}{a-d} \binom{k}{b-d} p^{-d}(1-p)^{-\Phi(a, b)}.
			\end{equation}
			The following lemma is essential to our argument.
				\begin{LEMMA} \label{lemma:maxdensity1}
					For any $k \in \mathbb{N}$ and  $a, b \in [k]$, \ 
                $ \displaystyle \frac{\Phi(a, b)} {k (a+b)} \leq 1 - \frac{\sqrt{2} }{2}.$
				\end{LEMMA}
				\noindent From its proof, which we defer to the end of this section, it will be clear that the maximum value of $T_{a, b,d }$ is achieved at 
				 $a = b = \frac{\sqrt{2}}{2}k$.  
				
				We now split into two cases according to the value of $a+b$. Firstly, assume that $a+b \geq \frac {24}{\epsilon p}.$ 
				    In this case, we trivially bound the \tk{product of the} three
                                    remaining binomial coefficients in \eqref{eq:tabd} by
                                    $2^{3k}$. The whole purpose of analysing the term $T_{a,b,d}$
                                    was do gain an extra factor of $k^{-d}$ compared to the trivial
                                    bound on the number of vertex orderings in $(\mb s, \mb
                                    t)$. \tk{This gives}
				\begin{align*}
					T_{a, b, d} &\leq 4^{a+b} 2^{3k} (np)^{-(a+b)} (kp)^{a+b-d}e^{\left(p+p^2 \right) \Phi(a, b)}, \\
					\sum_{d=0}^a T_{a, b, d} &\leq 4^{a+b+1} \cdot 2^{3k} (np)^{-(a+b)} (kp)^{a+b}e^{\left(p+p^2 \right) \Phi(a, b)}.							
				\end{align*}
					For the second inequality, we just summed over $d$ and  used $\sum_{d=0}^a (kp)^{-d} \leq \sum_{d=0}^a \left(2 \log (np) \right)^{-d}<4$. Using Lemma~\ref{lemma:maxdensity1} and $\frac{k}{a+b} \leq \frac{\epsilon}{6} \log (np)$, we get 
			\begin{align*}
				\sum_{d=0}^a T_{a, b, d} & \leq \left( 8kp \cdot 2^{\frac{\epsilon}{2}\log(np)} (np)^{-1}  e^{\left(1- \frac{\sqrt{2}}{2}\right)k \left(p+p^2\right) }  \right)^{a+b}\tk{.}
			\end{align*}
				Let $n$ be large enough so that $p \leq 0.1\epsilon$. \tk{Now} $kp =  \left( 1- \epsilon \right) \left(2+ \sqrt{2} \right) \log (np)$ gives $e^{\left(1- \frac{\sqrt{2}}{2}\right)k \left(p+p^2\right)} \leq e^{(1-\epsilon)(1+0.1\epsilon)\log(np)}\leq (np)^{(1-0.9\epsilon)}$, and hence
			\begin{align*}
                          \sum_{d=0}^a T_{a, b, d}  \leq
                               \tk{ \left(8(1-\epsilon)(2+\sqrt2)\log(np)(np)^{\left(\frac12\log 2-0.9\right)\epsilon}\right)^{a+b}} \leq (np)^{-\frac{\epsilon}{4}(a+b)}
			\end{align*}
				for large enough $n$. Summing over $a$, $b$, we get
				\begin{equation} \label{eq:tabd for large ab}
					\sum_{\substack{
						a, b \in [k]\\
						a+b \geq \frac{\epsilon}{24p} }}
						\sum_{d=0}^a T_{a, b, d} \leq \sum_{a=1}^k \sum_{b=1}^k (np)^{-\frac{\epsilon}{4}(a+b)} \leq \left ( \sum_{a=1}^k (np)^{-\frac{\epsilon a}{4}} \right)^2 \longrightarrow 0.
				\end{equation}
				        
           In the second case, $a+b < \frac{\tk{24}}{\tk{\epsilon}p}$, we bound the binomial coefficients by
           \tk{$\binom ki \leq k^i$, $i\in\{d,\,a-d,\,b-d\}$}. Then
			$$T_{a, b, d} \leq 4^{a+b}n^{-(a+b)}k^{2a+2b-2d} p^{-d}(1-p)^{-ab} .$$
			Assuming $p < \frac 12$, we use the inequality $(1-p) \geq e^{-p-p^2} \geq e^{-2p}$, which implies $(1-p)^{-ab} \leq e^{2pab}\leq e^{\frac{p(a+b)^2}{2}}$. Furthermore, $kp > 1$, so
			$$T_{a, b, d } \leq \left( 4 n^{-1}k^2 e^{\frac{p(a+b)}{2}} \right)^{a+b}.$$
            From the condition $p \geq \frac{\log^2 n}{\sqrt{n}}$, it follows that $n^{-1}k^2 \leq n^{-1}\cdot \frac{16n}{\log^4 n} \log^2 (np) \leq \frac{16}{ \log^2n}$. Moreover, $e^{\frac{p(a+b)}{2}} \leq e^{\frac{12}{\epsilon}}$, so altogether,
$T_{a, b, d } \leq \left( \frac{1}{\log n}\right)^{a+b}$ for large enough $n$.
			Summing up, 
				\begin{equation*}
					\sum_{\substack{
						a, b \in [k]\\
						a+b < \frac{\epsilon}{24p} }}
						\sum_{d=0}^a T_{a, b, d} \leq \sum_{a=1}^k \sum_{b=1}^k \tk{(a+1)}( \log n)^{-(a+b)} \leq \left ( \sum_{a=1}^k \tk{(a+1)}( \log n)^{-a} \right)  \left ( \sum_{b=1}^k ( \log n)^{-b} \right) \longrightarrow 0.
				\end{equation*}
            We conclude that
             $\frac{\var{X_{k}}}{\left(\er{X_{k}} \right)^2} = o(1)$, and hence $X_{k} >0$ with high probability. Upon this event, Lemma~\ref{lemma:witness} implies $Z(G) \leq n-k$. 
     \end{proof}

      Now we prove Lemma~\ref{lemma:maxdensity1}, which is essentially finding the induced subgraph of a $k$-witness with a minimum expected number of copies in $\gnp$.
        \begin{proof}[Proof of Lemma~\ref{lemma:maxdensity1}]
            Let $k$ be fixed. Recall that we are trying to maximise $\frac{\Phi(a, b)} {k (a+b)}$ over all $a$, $b$, where $\Phi(a, b) =  \max_{A, B \subs [k],\, |A| = a, \, |B| = b} \left | \left \{ (i, j) \in A \times B : i<j \right \} \right |.$
            
            We first fix $|A| = a$, $|B| = b$ and $|A \cap B| = g$.  If $A$ and $B$ are a selection of rows and columns of a $k \times k$ matrix, then $g$ denotes the number of diagonal entries at the intersection of selected rows and columns. Define $\phi(A, B) =  \left | \left \{ (i, j) \in A \times B : i<j \right \} \right |$. We claim that $\phi(A, B) \leq ab - \frac{(g+1)g}{2}$. To see this, note that each element $(c_1, c_2)$ with $c_1, c_2 \in A\cap B$ and $c_1 \geq c_2$ is contained in $A \times B$, but not counted by $\phi$. There are $\frac{(g+1)g}{2}$ such pairs (in fact, such  $(c_1, c_2)$ are indices of the subdiagonal matrix entries selected by $A$ and $B$).
            
            Now we minimise $g$ for fixed $a$ and $b$. From the identity $|A \cup B| + |A \cap B| = |A|+ |B|$ and $|A \cup B| \subs [k]$, we get $g \geq a+b -k$, so $\phi(A, B) \leq  ab - \frac{(a+b-k+1)(a+b-k)}{2} $. Taking the maximum over $A$ and $B$ and 
using the AM-GM inequality, we get
            \begin{align*}
                \Phi(a, b) &\leq ab - \frac{(a+b-k+1)(a+b-k)}{2} \leq \left( \frac{a+b}{2} \right)^2 - \frac{(a+b-k)^2}{2} \\
                \frac{\Phi(a, b)} {k (a+b)} &\leq \frac{1}{k(a+b)} \left(\left( \frac{a+b}{2} \right)^2 - \frac{(a+b-k)^2}{2}  \right).
            \end{align*}
            
            We substitute $a+b = 2\alpha k$, so that the problem reduces to maximising the function $f(\alpha) = \frac{1}{4\alpha} \left(2\alpha^2 - (2\alpha-1)^2 \right) $ for $0 < \alpha \leq 1$.  This is a simple calculus exercise, but we provide details for the sake of transparency.
            $
                f'(\alpha) = \frac{1}{4 \alpha^2}(-2\alpha^2+1)
            $,
            so $f$ attains its local maximum at $\alpha_{0} = \frac{\sqrt{2}}{2}$. Evaluating $f$ at $\alpha_{0}$, we get
            $$\Phi(a, b) \leq \max_{\alpha} f(\alpha)  = f(\alpha_0)
            = 1 - \frac{ \sqrt{2}}{2} 
            $$
            for all $a, b \subs [k]$. Note that the equality is asymptotically attained when $A = [\ell]$, $B = [k-\ell+1, k]$ with $\ell = \lfloor \frac{\sqrt{2}k}{2} \rfloor$, which we have used in Lemma~\ref{lemma:gnp_fmm} (the 0-statement of Theorem~\ref{thm:forcing_gnp}).
\end{proof}

        \section{Spectral bounds}
        
			In this  section we discuss the bounds on the zero forcing number in terms of the graph eigenvalues.
			The study of spectral properties and their relation to other graph parameters is an established area of research with many diverse techniques and applications, surveyed for example in the monograph of Godsil and Royle~\cite{gr01}. 
			One of the earliest results of this type is Hoffman's bound on the independence number of a graph. Namely, let $G$ be an $n$-vertex $d$-regular graph, and let $\lambda_{\min}$ denote its smallest eigenvalue. Hoffman proved that then $G$ contains no independent set of order larger than $\frac{-\lambda_{\min} n}{d-\lambda_{\min}}$. Note that since the trace of the adjacency matrix of a graph is zero, $\lambda_{\min}$ is negative. There are many examples showing the bound to be tight.
        
            We establish an analogue of Hoffman's bound for the zero forcing number, showing that $Z(G) \geq n \left(1+\frac {2\lambda_{\min}}{d-\lambda_{\min}}  \right).$
            To prove this result we use the following well-known estimate on the edge distribution of a graph in terms of its eigenvalues. Part (ii) is provided in, e.g.,~\cite{ks06}, whereas the variant (i) follows from the same proof. For a graph $G=(V, E)$ and two sets $U, W \subs V$, denote the number of edges with one endpoint in $U$ and the other one in $W$ by $e(U, W)$. Any edge with both endpoints in $U \cap W$ is counted twice. Recall that an $\ndl$ graph is a $d$-regular $n$-vertex graph in which all eigenvalues but the largest one are at most $\lambda$ in absolute value.
            
            \begin{THM} \label{thm:ks:edgedist}
                Let $G$ be an $\ndl$-graph, and denote its smallest eigenvalue by $\lambda_{\min}$. Then for any two vertex subsets $U, W$ of $G$,
                \begin{enumerate}
                    \item 
                $ \displaystyle{   \frac{d|U||W|}{n}  - e(U, W) \leq\
                -\lambda_{\min} \sqrt{|U||W|\left(1 - \frac{|U|}{n} \right)\left(1 - \frac{|W|}{n} \right)}.}$
                    \item $\displaystyle{  \left | \frac{d|U||W|}{n}  - e(U, W)  \right | \leq\
                \lambda \sqrt{|U||W|\left(1 - \frac{|U|}{n} \right)\left(1 - \frac{|W|}{n} \right)}.}$
                \end{enumerate}
            \end{THM}
            
            \begin{proof}[Proof of Theorem~\ref{thm:zeroforcing_pseudo} (i)]
                Let $(\mb s, \mb t)$ be a $k$-witness in $G$,
                and $k = 2\mu n$. Note that there are no edges between the sets $U = \{s_1, s_2, \dots, s_{\mu n}\}$ and $W = \{t_{\mu n +1}, \dots, t_k \}$. Hence, using Theorem~\ref{thm:ks:edgedist}, we get
                    \begin{align*}
                        0  = e(U, W) &\geq d \mu^2 n + \lambda_{\min} \mu n (1 - \mu)\\           
                        -\lambda_{\min} & \geq d\mu - \lambda_{\min}\mu\\
                        \mu &\leq \frac{-\lambda_{\min}}{d - \lambda_{\min}}.
                    \end{align*}
                
                Hence the largest witness in $G$ has order at most $\frac{-2\lambda_{\min} n}{d - \lambda_{\min}}$, which by Lemma~\ref{lemma:witness} implies 
                $$Z(G) \geq n \left(1+\frac {2\lambda_{\min}}{d-\lambda_{\min}}  \right).\qedhere$$
                
            \end{proof}
            Surprisingly, the additional factor of two in the above-mentioned bound that looks like an artefact of the proof, turns out to be necessary, and the result of Theorem  (i) is shown to be tight by the following example. 
            
            \begin{PROP} \label{prop:linegraph}
                For any  $D \geq 2$, and for infinitely many values of $N$, there exists an $N$-vertex $D$-regular graph $G^*$ whose smallest eigenvalue is $\lambda_{\min} = -2$, and which satisfies $N - Z\left(G^* \right)\geq \frac{4N}{D+2}-2$.
            \end{PROP}
            
            \begin{proof}
                Let $G$ be an $n$-vertex $d$-regular graph which contains a Hamilton cycle consisting of edges $e_1, e_2, \dots, e_n$ in this order. Clearly, such graphs do exist.  Let $G^*$ be the line graph of $G$, that is, $G^*$ has the vertex set $E(G)$ with two vertices adjacent if the corresponding edges in $G$ share a vertex. Then $G^*$ has $N = \frac{nd}{2}$ vertices and is $D$-regular with $D = 2d-2$. Moreover, Hoffman~\cite{hoffman} has observed that the smallest eigenvalue of $G^*$ is -2.
                
                Note that the vertices in $G^*$ corresponding to $e_1, e_2, \dots, e_n$  form an induced cycle. This implies that $V(G^*) \setminus \{ e_3, e_4, \dots, e_n\}$ is a zero forcing set in $G^*$. Namely, the vertex $e_i$ forces $e_{i+1}$ for $i = 2, 3, \dots, n-1$. This zero forcing set has order $N-(n-2)$. Finally, notice that in $G^*$, we have
                $$-\frac {2\lambda_{\min}N}{d-\lambda_{\min}} = \frac{4 \cdot \frac{nd}{2}}{2d-2 + 2} =n.\qedhere$$
            \end{proof}

            Next, we turn our attention to  the second part of Theorem~\ref{thm:zeroforcing_pseudo}, which says that any 
            $\ndl$-graph $G$  satisfies
            $$Z(G) \leq n \left(1-\frac {1}{2(d-\lambda)} \log \left(\frac{d-\lambda}{2\lambda +1} \right) \right).$$ In particular, if $\lambda = d^{1-\epsilon}$ for some $\epsilon >0$, then $n- Z(G) = \Omega \left(\frac{n \log d}{d} \right)$.            
      
            \begin{proof}[Proof of Theorem~\ref{thm:zeroforcing_pseudo} (ii)] 
                We greedily construct a witness. In each step $i$, we will select vertices $s_i, t_i \in U_{i-1}$ and a set $U_i \subs U_{i-1}$. Start with $U_0 = V$, the vertex set of $G$. Assuming that the steps $1, \dots, i-1$ were executed, let $s_i$ be any vertex in $U_{i-1}$ satisfying $1 \leq \deg_{G[U_{\tk{i-1}}]}(s_i) \leq (d- \lambda)\frac{|U_{i-1}|}{n}+ \lambda$. We fix $t_i$ to be any neighbour of $s_i$, and set $U_i = U_{i-1} \setminus N_G(s_i)$ (note that $N_G(s_i)$ contains $s_i$). The algorithm continues as long as $|U_i| > \frac{\lambda n}{d+ \lambda}$. Denote the final number of steps by $k$.
                
                It is clear that such $k$-tuples $(\mb s,\mb t)$ form a witness. We will show that there is a choice for $s_i$ throughout the algorithm, and that $k \geq \frac {n}{2(d-\lambda)} \log \left(\frac{d-\lambda}{2\lambda +1} \right) $.
                
                \begin{CLAIM} \tk{If $\lvert U_i\rvert>\frac{\lambda n}{d+\lambda}$, then} the induced subgraph $G[U_i]$ contains a vertex $u$ satisfying
                $1 \leq \deg_{G[U_i]}(u) \leq (d- \lambda)\frac{|U_i|}{n}+ \lambda$.
                \end{CLAIM}
                Suppose that some set $U_i$ does not satisfy the Claim. Since $|U_i| > \frac{\lambda n}{d+ \lambda}$, it is not an independent set by Theorem~\ref{thm:ks:edgedist}. Therefore, removing all the isolated vertices in $G[U_i]$, we get a non-empty set of vertices $W \subs U_i$ in which every vertex $u$ satisfies 
                $\deg_{G[W]}(u) \tk{>} (d- \lambda)\frac{|U_i|}{n}+ \lambda$. In particular,
                $$\tk{e}(W, W) > 2 \cdot \frac{|W|}{2} \left( (d- \lambda)\frac{|U_i|}{n}+ \lambda \right),$$
                recalling that we are counting each edge in $E(W, W)$ twice. On the other hand, Theorem~\ref{thm:ks:edgedist} implies that
                $$\tk{e}(W, W) \leq {|W|}\left(\frac{d|W|}{n}+ \tk{\lambda} \left(1 - \frac{|W|}{n} \right) \right) \leq |W| \left( (d- \lambda)\frac{|U_i|}{n}+ \lambda\right). $$
                We reached a contradiction, which completes the proof of the Claim.
                
                Now denote $a_i = \frac{|U_i|}{n}$. By construction, $a_0 = 1$ and
                $$a_{i+1} \geq a_i -  \frac{d-\lambda}{n} \cdot a_i - \frac{\lambda + 1}{n}.$$
                
                \begin{CLAIM}  
                    For $i \leq \frac{n}{2(d- \lambda)}\log \frac{d-\lambda}{2 \lambda +1}$,
                    $a_i \geq \frac{\lambda}{d+\lambda}$.
                \end{CLAIM}
                
                It is not hard to show by induction that for all $i$,
                $$a_i \geq \left(1 + \frac{\lambda +1}{d-\lambda} \right) 
                \left(1 - \frac{d-\lambda}{n}\right) ^i
                - \frac {\lambda +1}{d- \lambda}.$$

                    Now we estimate $a_i$, ignoring the constant $\left(1 + \frac{\lambda
                        +1}{d-\lambda} \right)$ and using the inequality $1 - \frac{d-\lambda}{n}
                    \geq e^{-\frac{2(d-\lambda)}{n}}$ for $\tk{\frac{d-\lambda}{n}} <
                    \frac12$. \tk{This gives}
                    \begin{align*}
                        a_i &\geq e^{-\frac{2(d-\lambda)}{n} \cdot \frac{n}{2(d-\lambda)}\log \left (\frac{d-\lambda}{2 \lambda +1} \right)}  - \frac {\lambda +1}{d- \lambda} 
                        = \frac{2 \lambda +1}{d- \lambda} - \frac{\lambda + 1}{d- \lambda}
                        \tk{=} \frac{\lambda}{d+\lambda},
                    \end{align*}
                    as required.
                    We get that the algorithm continues for at least $k = \frac{n}{2(d- \lambda)}\log \frac{d-\lambda}{2 \lambda +1}$ steps, so
                    $$Z(G) \leq n\left(1- \frac{1}{2(d- \lambda)}\log \frac{d-\lambda}{2 \lambda +1} \right).\qedhere$$
                
            \end{proof}    

            To show that this bound is tight up to a constant factor, we exhibit a sequence of $\ndl$-graphs $G_m$ with $\lambda = O \left( \sqrt{d} \right)$ whose forcing number is at least $n \left(1 -  \frac{  \log_2 d }{2d}  +o(1)\right)$. 
            We use the following construction from~\cite[Section 3]{ks06}. For an odd integer $m$, the vertices of $G_m$ are all binary vectors of length $m$ with an odd number of ones except for the all-one vector. Two distinct vertices are adjacent iff the inner product of the corresponding vectors is 1 modulo 2. This graph has $n_m = 2^{m-1}-1$ vertices, degree $d_m = \frac {n_m-3}{2}$, and second largest eigenvalue $\lambda (G_m) = 1 + 2^{\frac{m-3}{2}} = O \left(\sqrt{d_m} \right)$. It is easy to check that if $(\mb s, \mb t)$ is a $k$-witness in $G_m$, then the vectors corresponding to $t_1, t_2, \dots, t_k$ are linearly independent, and therefore $k \leq m = (1+o(1))\log_2 n_m = \frac{(1+o(1))n_m}{2d_m}\log_2 d_m$. This implies the required bound on $Z(G_m)$.

	        \section{Concluding Remarks}
			\begin{itemize}
				    \item Theorem~\ref{thm:forcing_gnp} can be extended to $p \tk{=} \omega \left(n^{-1} \right)$. The proof combines our second moment estimates with Talagrand's inequality, along the lines of~\cite[Theorem 7.4]{jlr}, which finds independent sets of order $p^{-1}\log(np)$ in $\gnp$. Since this proof gives no additional insight, we put it into the appendix.
					\item It would be interesting to study \tk{the} minimum rank of quasirandom graphs. Recall that the minimum rank of the random graph $\gn$ is bounded away from $0$, that is, $mr(\gn) \geq 0.14 n$ with high probability (see, e.g.~\cite{hhms}). We wonder if this also holds for $\ndl$-graphs when $d$ is linear in $n$ and $\lambda=o(n)$.
				\item Another question would be to find the zero forcing number of
                                  the random regular graph $G_{n,d}$, which is a graph chosen
                                  uniformly at random from all $n$-vertex $d$-regular graphs. We are
                                  interest\tk{ed in} $G_{n,d}$ for a large constant $d$ and $n \rightarrow \infty$. 
				A greedy argument (see, e.g.~\cite{acdp}) shows that $Z\left(G_{n,d}\right) \leq \tk{n}\left(1-\frac{1}{d-1} \right)$ deterministically, whereas Theorem~\ref{thm:zeroforcing_pseudo} (ii) implies that for large $d$, with high probability, $Z \left( G_{n, d} \right) \leq n\left(1- \frac{\log d}{4d} \right)$. This follows from the fact that with high probability, $G_{n, d}$ is an $\ndl$-graph with $\lambda \leq 3\sqrt{d}$ (see, e.g.,~\cite{friedman}). The lower bound, $Z \left( G_{n, d} \right) \geq n\left(1- \frac{40\log d}{d} \right)$,  is an immediate consequence of the fact that with high probability, $G_{n,d}$ contains edges between any two sets $S$, $T$ with $|S|,\, |T| \geq \frac{20n \log d }{d}$ (see, e.g.,~\cite[Lemma 3.6]{kls}). It would be interesting to find the correct constant. 
			\end{itemize}

		\appendix
         \section{The random graph with small \emph{p}} 
       		To extend Lemma~\ref{lemma:gnp_smm} for small $p$ we use Talagrand's Inequality (see, e.g.,~\cite[Theorem 2.29]{jlr}). 
			\begin{LEMMA}
				Assume $ \omega\left(\tk{n^{-1}} \right)=p< \frac{\log^2 n}{\sqrt{n}}$. Let the vertex set $V$ of $\gnp$ be partitioned into $V_1$ and $V_2$ with $|V_1| = \lfloor \frac n2 \rfloor$ and $|V_2| = \lceil \frac n2 \rceil$, and let $k_{-\epsilon} = \frac{(1-\epsilon)(2+\sqrt{2})}{p}\cdot \log (np)$  with $0< \epsilon < \frac 12$.  With high probability, $\gnp$ contains a divided $k_{-\epsilon}$-witness.
			\end{LEMMA}

			\begin{proof}
			 Denote by $w(G)$ the order  of the largest divided witness. For a given $\epsilon$, we will actually show that $w\left(\gnp \right) \geq k_{-2\epsilon} = \frac{(1-2\epsilon)(2+\sqrt{2})}{p}\cdot \log (np)$ whp. The first step is a second moment lower bound on the probability $\pr{w\left(\gnp \right) \geq k_{-\epsilon}}$. For now, we write $k = k_{-\epsilon}$. 
			The proof is identical to the proof of Lemma~\ref{lemma:gnp_smm} down to equation \eqref{eq:tabd for large ab}. The case $a+b \geq \frac{24}{\epsilon p}$ remains unchanged (and that is the case which determines $k$). Hence we assume $a+b < \frac{24}{\epsilon p}$.
 Equation \eqref{eq:tabd} from the proof of Lemma~\ref{lemma:gnp_smm} implies
\begin{align*}
	T_{a,b,d} &\leq \left(\frac{2}{n} \right)^{a+b} k^{a+b-d} p^{-d} \binom{k}{d}\binom{k}{a-d} \binom{k}{b-d}(1-p)^{-ab} \\
				& \leq \left(\frac{2k}{n} \right)^{a+b} \binom{k}{d}\binom{k}{a}\binom{k}{b} (1-p)^{-ab}\tk{.}
\end{align*}
	Denote $a+b = 2u$ and define $S_u := \left( \frac{2k}{n} \right)^{2u} \binom{k}{u}^3(1-p)^{-u^2}$. It is clear that $T_{a,b,d} \leq S_{\frac{a+b}{2}}=S_u$, so we will maximise $S_u$ for $u \in \left[1, \frac{12}{\epsilon p} \right]$. We compute
$$ \frac{S_{u+1}}{S_u} = \left(\frac{2k}{n} \right)^2 \cdot \left( \frac{k-u}{u+1} \right)^3(1-p)^{-2u-1}. $$

We will use the following two claims.
\begin{CLAIM} \label{claim:second derivative of Su}
	  $\frac{S_{u+1}}{S_u}$ has at most one local extremum on $\left[ 1, \frac{12}{\epsilon p} \right]$. 
\end{CLAIM}
\begin{CLAIM} \label{claim:values of Su}
	For $u_1 = 2k \left( \frac{k}{n} \right)^{\frac{2}{3}}$, $\frac{S_{u_1+1}}{S_{u_1}}<1$. Moreover, $\frac{S_2}{S_1} \geq 1 $.
\end{CLAIM}
	The two claims imply that the maximum of $S_u$ lies in the interval $[1, u_1]$. For, there is at most one point $u$ where $\frac{S_{u+1}}{S_u}$ switches from larger than one to smaller than one. The second claim implies that such a point lies in the interval $[1, u_1]$. We will prove the two claims later. Now we finish the proof assuming the claims, that is, we bound $S_u$ for $u \leq u_1$. 
	\begin{align*}
	    S_u &= \left( \frac{2k}{n} \right)^{2u} \binom{k}{u}^3(1-p)^{-u^2} 
	        \leq \left( \frac{4k^2}{n^2}\cdot \frac{e^3k^3}{ u^3} \right)^u e^{2pu^2}.
	\end{align*}
	We use the fact that
	\begin{enumerate}
	    \item $2pu_1 \leq 2pk \leq 4\log(np)$, and
	    \item the function $u \mapsto \left(\frac{4e^3k^5}{n^2u^3} \right)^u$ is increasing for $u^3 \leq \frac{4e^2 k^5}{n^2}$, and in particular on $[1, u_1]$.
	\end{enumerate}
	
	\begin{align*}
	    \log S_u & \stackrel{\textrm{(ii)}}{\leq} u_1 \left(\log \left(\frac{4k^2}{n^2}\cdot \frac{e^3n^2}{8k^2}\right)  +2pu_1\right) 
	    \stackrel{\textrm{(i)}}{\leq} u_1 \left(3 + 4\log(np) \right)
			\leq k \cdot 10 \log(np) \cdot \left( \frac kn \right) ^\frac{2}{3} \leq k (np)^{-\frac 12}.
	\end{align*}
	The third inequality, where $u_1 \tk{\ll} k$ is really used, is crucial in the calculation. We conclude that for all $a, b, d$ such that $a+b < \frac{24}{\epsilon p}$,
	\begin{equation*}
	    T_{a, b, d} \leq e^{k(np)^{-\frac 12} }\tk{.}
	\end{equation*}
	
	Summing over all such $a, b, d$ and using (\ref{eq:tabd for large ab}) we get
	\begin{equation*}
	    \frac{\var{X_k}}{\left( \er{X_k} \right)^2} \leq 
	    \sum_{\substack{
						a, b \in [k]\\
						a+b \geq \frac{\epsilon}{24p} }}
						\sum_{d=0}^a T_{a, b, d} 
		+ \sum_{\substack{
						a, b \in [k]\\
						a+b < \frac{\epsilon}{24p} }}
						\sum_{d=0}^a T_{a, b, d}						
		\leq k^3 e^{k(np)^{-\frac 12} } +o(1)
		\leq e^{2k(np)^{-\frac 12} }\tk{.}
	\end{equation*}
    We will apply a stronger form of Chebyshev's inequality, which reads $\pr{X_k > 0 }\geq  \frac{\left( \er{X_k} \right)^2}{\er{X_k^2}}$. For details, refer to~\cite[Remark 3.1]{jlr}. 
    \begin{align*}
        \frac{\er{X_k^2}}{(\er X_k)^2}   &= \frac{\var{X_k}}{(\er{X_k})^2}+1 \leq e^{4k(np)^{-\frac 12} }, \quad \text{so} \\
        \pr{X_k >0} &\geq e^{-4k(np)^{-\frac 12} }.
    \end{align*}
	To show concentration of the order of the largest divided witness $w\left( \gnp \right)$, we apply Talagrand's Inequality. The random graph $\gnp$ is modelled using vertex exposure. Formally, we fix an ordering of the vertices $v_1, v_2, \dots v_n$, and     define mutually independent random variables $\left( Z_i \right)_{i \in [n]}$, where $Z_i$ exposes the backward edges from the vertex $v_i$. Then the  $w\left(\gnp\right)$ is a function of $Z_1, \dots, Z_n$. This function is 1-Lipschitz, that is, if graphs $G$ and $G'$ differ only at the vertex $v_i$, then $|w(G)-w(G')| \leq 1$. Moreover, whenever $w(\gnp) \geq k$, there exist $2k$ certificate vertices, namely the vertices of a divided $k$-witness, which are responsible for the fact that $w(\gnp) \geq k$. 
	Hence we may apply~\cite[Theorem 2.29]{jlr} with $\psi(k) = 2k$ in their notation. Recalling that $k_{-\epsilon} = \frac{(1-\epsilon)(2+\sqrt{2})}{p}\cdot \log (np)$, we have
		$$\pr{w(G) \leq k_{-2\epsilon} } \pr{ w(G) \geq k_{-\epsilon}} \leq e^{-\frac{\left( k_{-\epsilon} -k_{-2\epsilon} \right)^2}{8k_{-\epsilon}}} \leq e^{-\frac{\epsilon^2 k_{-\epsilon}}{8}}.$$
		We have shown that $\pr{w(G) \geq k_{-\epsilon}} \geq e^{-4k_{-\epsilon}(np)^{-\frac 12}}.$
	Putting this together with the previous inequality and taking $np > \left(\frac{10}{\epsilon} \right)^4$,
		$$\pr{w(G) \leq k_{-2\epsilon}} \leq e^{-\frac{\epsilon^2 k_{-\epsilon}}{8} + 4k_{-\epsilon}(np)^{-\frac 12}} \leq e^{-\frac{\epsilon^2 k_{-\epsilon} }{16}} \longrightarrow 0,$$
as required.

\begin{proof}[Proof of Claim~\ref{claim:second derivative of Su}]
	We differentiate $f(u) =  \left( \frac{k-u}{u+1} \right)(1-p)^{-2u/3}$, the part of $  \frac{S_{u+1}}{S_u} $ which is dependent on $u$, raised to power $\frac 13$ to simplify the calculation. This is valid since the function $x \mapsto x^3$ is increasing. Let us replace $(1-p)$ by $e^{-q}$, where $q = -\log(1-p)$.
	$$f'(u) = e^{2qu/3} \left( \frac{-1}{u+1} -  \frac{k-u}{(u+1)^2} + \frac{k-u}{u+1}\cdot \frac{2q}{3} \right)$$
	$$f'(u) = 0 \quad \Leftrightarrow \quad k+1 + \frac{2q}{3}\left( -ku -k + u^2 +u \right)=0$$
This is a quadratic equation in $u$. Viet\'e's formulae give that if $u_-$ and $u_+$ are solutions, then $u_- + u_+ = k-1$. Therefore the larger solution, say $u_+$, is at least $\frac{k-1}{2}$. Since $\frac{12}{\epsilon p} < \frac{k-1}{2}$ provided that $np$ is sufficiently large, $u_-$ is the only potential local extremum in $\left[ 1, \frac{12}{\epsilon p} \right]$.	
\end{proof}
\begin{proof}[Proof of Claim~\ref{claim:values of Su}]
	Take $u_1$ as in the statement. Then $\frac{k^3}{u_1^3}= \frac 18 \left( \frac nk \right)^2$
$$ \frac{S_{u+1}}{S_u} \leq \left(\frac{2k}{n} \right)^2 \cdot 2^{-3} \cdot \frac{n^2}{k^2} \cdot e^{8pk \left( \frac kn \right)^{2/3}} = (1+o(1))\frac 12.$$
	Furthermore,
	$$\frac{S_2}{S_1} = \left( \frac{2k}{n}^2 \right) \left( \frac{k-1}{2} \right)^3 (1-p)^{-2} \geq \frac{k^5}{10n^2}\cdot \frac 14 > 1.\qedhere$$
\end{proof}
\tk{This concludes the proof of the lemma.}
\end{proof}    
        
\end{document}